\documentclass[leqno,12pt]{amsart} 
\usepackage[top=30truemm,bottom=30truemm,left=25truemm,right=25truemm]{geometry}
\usepackage{amssymb}
\usepackage{amsmath}
\usepackage{amsthm}
\usepackage{amscd}
\usepackage{mathrsfs}
\usepackage{graphicx}
\usepackage[dvips]{color}
\usepackage[all]{xy}
\usepackage{url}
\usepackage{comment}
\theoremstyle{plain} 
\newtheorem{theorem}{\indent\bf Theorem}[section]
\newtheorem{lemma}[theorem]{\indent\bf Lemma}

\newtheorem{claim}[theorem]{\indent\bf Claim}
\newtheorem{conjecture}[theorem]{\indent\bf Conjecture}
\theoremstyle{definition} 

\newcommand{\twodimvector}[2]{
	\begin{pmatrix}
		#1\\#2
	\end{pmatrix}
}

\begin{document}
	
	\title[]{Subharmonic variation of Azukawa pseudometrics for balanced domains}
	
	\author[G. Hosono]{Genki Hosono} 
	
	%
	\address{
		Graduate School of Mathematical Sciences, The University of Tokyo \endgraf
		3-8-1 Komaba, Meguro-ku, Tokyo, 153-8914 \endgraf
		Japan
	}
	\email{genkih@ms.u-tokyo.ac.jp}
	\maketitle
	\begin{abstract}
	We consider the subharmonicity property of the logarithm of Azukawa pseudometrics of pseudoconvex domains under pseudoconvex variations.
	We prove that such a property holds for the variation of balanced domains.
	We also give a non-balanced example.
	The relation of the volume of Azukawa indicatrix and the estimate in the Ohsawa-Takegoshi $L^2$-extension theorem is also discussed.
	\end{abstract}
	
	\section{Introduction}
	Let $\Omega \subset \mathbb{C}^n$ be a bounded hyperconvex domain.
	For a fixed point $w \in \Omega$, {\it the pluricomplex Green function} $g_{\Omega, w}$ on $\Omega$ with a pole $w$ is defined as
	$$g_{\Omega, w}:= \sup\{u \in PSH(\Omega): u<0 \text{ and } u - \log |\cdot - w| \text{ is bounded above near }w  \}.$$
	It is known that $g_{\Omega,w}$ itself also satisfies the condition in the right-hand side.
	Pluricomplex Green functions are pluripotential theoretic generalizations of Green functions, which are the solutions of the {Laplace equation} 
	$$\left\{
	\begin{array}{cc}
	\Delta g_{\Omega, w} = \delta_w &\text{ on } \Omega\\
	g_{\Omega,w} = 0 &\text{ on } \partial \Omega
	\end{array}\right.$$
	where $\delta_w$ means the Dirac measure concentrated at $w$.
	Similarly, pluricomplex Green functions are the solutions of the {\it complex Monge-Amp\`ere equations} with a logarithmic pole
	$$\left\{
	\begin{array}{cc}
	(dd^c g_{\Omega, w})^n = \delta_w &\text{ on } \Omega\\
	g_{\Omega,w} = 0 &\text{ on } \partial \Omega\\
	g_{\Omega,w} = \log |\cdot - w| + O(1) & \text{ near }w.
	\end{array}\right.$$
	Due to the non-linearity of the Monge-Amp\`ere operator, analysis of pluricomplex Green functions is much more difficult than that of Green functions.
	
	To analyze infinitesimal behavior of pluricomplex Green function near its pole, it is useful to consider (the logarithm of) {\it the Azukawa pseudometric}. For a fixed point $w \in \Omega$ and a vector $X \in \mathbb{C}^n$, it is defined as
	$$A_{\Omega, w}(X):=\limsup_{\lambda \to 0}(g_{\Omega, w}(w + \lambda X) - \log |\lambda|).$$
	When $n=1$ and $X=1$, it measures the difference between the Green function and the logarithm:
	$$A_{\Omega, w}(1) = \limsup_{\lambda \to 0}(g_{\Omega,w}(w + \lambda) - \log |\lambda|).$$
	This value is called as the {\it Robin constant}.
	
	In this note, we are interested in the variational theory of such values.
	Some variational results have been obtained by many authors. In \cite{Yam}, Yamaguchi proved that Robin constants are subharmonic under pseudoconvex variations:
	
	\begin{theorem}[\cite{Yam}]\label{thm:Robin-psh}
		Let $t \in \Delta \subset \mathbb{C}$ and let $D(t)$ be an unramified coveing domain over $\mathbb{C}$ with boundary $\partial D(t)$. Set $\mathbb{D} := \{(t,z): t \in \Delta, z \in D(t)\}$ and assume that $\mathbb{D}$ is pseudoconvex. Fix $\zeta$ such that $\zeta \in D(t)$ for every $t \in \Delta$. Then, the Robin constant $\lambda(t) := \lim_{z \to \zeta} g(z) - \log |z - \zeta|$ is a subharmonic function of $t$.
	\end{theorem}
	In \cite{MY}, by a formula on the Robin constant and the Bergman kernel, the subharmonicity of Bergman kernels under pseudoconvex variations is proved.
	Their result was widely generalized in \cite{Ber2}. In \cite{Ber2}, the Nakano positivity of the Hilbert space bundle of $L^2$-holomorphic functions was proved.
	
	Our aim in this note is to consider the pluripotential version of Theorem \ref{thm:Robin-psh}. Our ultimate goal is to prove the following conjecture:
	
	\begin{conjecture}\label{conj:azukawa-psh}
		Let $\widetilde{\Omega} \subset \mathbb{C}^n \times \Delta$ be a pseudoconvex domain.
		For each $t \in \Delta$, define $\Omega_t$ by $\{z \in \mathbb{C}^n: (z,t) \in \widetilde{\Omega} \}$. 
		Fix a point $w \in \mathbb{C}^n$.
		We assume that $\Omega_t$ is hyperconvex for every $t \in \Delta$ and $w \in \Omega_t$.
		Then,
		\begin{itemize}
			\item[(1)] For $(X,t) \in \Delta \times \mathbb{C}^n$, the function 
			$$A_t(X) := A_{\Omega_t, w}(X)$$
			is plurisubharmonic with respect to $(X,t)$.
			\item[(2)] For $t \in \mathbb{C}$, we set $I_t := \{X: A_t(X) < 0\}$. We denote the volume of $I_t$ by $V(t)$. Then the function $t \mapsto -\log V(t)$ is subharmonic.
		\end{itemize}
	\end{conjecture}
	
	Note that the following theorem shows that (1) implies (2) in Conjecture \ref{conj:azukawa-psh}, thus (2) is a weaker statement than (1).
	\begin{theorem}\label{thm:volume-of-log-hom-psh}
		Let $A(X,t)$ be a plurisubharmonic function on $\mathbb{C}^n \times \Delta$ such that $A(\lambda X, t) = A(X,t) + \log |\lambda|$. Set $I_t := \{X\in \mathbb{C}^n : A(X, t) < 0 \}$ for each $t$ and let $V(t)$ be the volume of $I_t$.
		Then the function $t \mapsto -\log V(t)$
		is subharmonic.
	\end{theorem}
	
	The {\it Azukawa indicatrix} is a set defined by $I_{\Omega,w} := \{X \in \mathbb{C}^n: A_{\Omega,w} (X) < 0 \}$.
	In Conjecture \ref{conj:azukawa-psh} (2), we consider the volume of the Azukawa indicatrix.
	The volume of the Azukawa indicatrix has been studied in relation of generalization of the Suita conjecture, for example, \cite{BZ1}. Very recently, using the Carath\'eodory-Reiffen pseudometric, B{\l}ocki and Zwonek proved the ``diagonal case'' of Conjecture \ref{conj:azukawa-psh} \cite[Proposition 13 and Theorem 14]{BZ2}, i.e. for a pseudoconvex domain $\Omega \subset \mathbb{C}^n$, the functions $\Omega \times \mathbb{C}^n \ni (z, X) \mapsto A_{\Omega, z}(X)$ and $\Omega \ni z \mapsto -\log{\rm Vol}(\{X: A_{\Omega, z}(X)<0\})$ are plurisubharmonic. 
	
	We do not know if Conjecture \ref{conj:azukawa-psh} is true in general, but we have a result for domains with some symmetry.
	A domain $\Omega \subset \mathbb{C}^n$ is said to be {\it balanced} if $\lambda \Omega \subset \Omega$ for every $\lambda \in \mathbb{C}$ with $|\lambda|\leq 1$.
	In this note, we prove that Conjecture \ref{conj:azukawa-psh} is true if every $\Omega_t$ is balanced.
	\begin{theorem}\label{thm:balanced}
		Conjecture 1.2 is true when every $\Omega_t$ is balanced and $w = 0$.
	\end{theorem}
	On a pseudoconvex balanced domain, one can write the pluricomplex Green function on it with a pole at the origin explicitly. We will use this explicit formula to prove Theorem \ref{thm:balanced}.
	
	We also provide an example consisting of domains that are not balanced with respect to their poles.
	\begin{theorem}\label{thm:example}
		Let $\varphi(t)$ be a subharmonic function on $\Delta$ with $\varphi < 0$. Define $\Omega_t := \{z \in \mathbb{C}^2 : |z|^2 < e^{-\varphi(t)}\}$. Let $w = (1,0) \in \mathbb{C}^2$. Then the function
		$$\mathbb{C}^2 \times \Delta \ni (X, t) \mapsto A_{\Omega_t, w}(X)$$
		is plurisubharmonic with respect to $(X, t)$.
	\end{theorem}
	
	We refer to an interesting topic on the Azukawa indicatrix, that is, a relation to the {\it Ohsawa-Takegoshi $L^2$ extension theorem}.
	In \cite{BL}, a Robin-type constant was used as an optimal coefficient of the $L^2$-extension theorem.
	We prove that we can use the Euclidean volume of the Azukawa indicatrix instead of a Robin-type constant. This change improves the estimate. Roughly speaking, it is because the Azukawa indicatrix has more information than the Robin-type constant has.
	This topic may have some relation with the higher dimensional Suita conjecture \cite{BZ1}, \cite{BZ2}.
	
	\section{Variation of sublevel sets of a log-homogeneous plurisubharmonic function}\label{sec:volume}
	In this section, we will prove Theorem \ref{thm:volume-of-log-hom-psh}.
	
	\begin{proof}[Proof of Theorem \ref{thm:volume-of-log-hom-psh}]
		By subtracting a sufficiently large constant from $A$, we can assume that the unit ball $B^{2n}$ in $\mathbb{C}^n$ is relatively compact in each $I_t$. First we represent $V(t)$ as an integration.
		\begin{claim}\label{clm:volume_integral}
			It holds that
			$$V(t) = C \int_{S^{2n-1}} e^{-2nA(\widehat{X}, t)} dS(\widehat{X}),$$
			where $S^{2n-1}$ is the unit sphere in $\mathbb{C}^n$, $\widehat{X}$ is a unit vector in $\mathbb{C}^n$, $dS$ is the volume form on $S^{2n-1}$ induced by the standard metric on $\mathbb{C}^n$ and $C$ is a constant which depends only on the dimension $n$.
		\end{claim}
		Fix a parameter $t \in \Delta$.
		For each $X \in \mathbb{C}^n \setminus \{0\}$, we denote the unit vector with the same direction as $X$ by $\widehat{X}$, i.e. $\widehat{X} = X / |X|$. Set $|X| =: r$.
		Then we have that
		$$A(X, t) < 0 \Longleftrightarrow A(r\widehat{X},t)<0 \Longleftrightarrow \log r < -A(\widehat{X},t).$$
		Thus it holds that
		$$I_t = \{X\in \mathbb{C}^n: r < e^{-A(\widehat{X}, t)}\}.$$
		We need to calculate the volume of this set. 
		The volume of infinitesimal cone in the $\widehat{X}$-direction is proportional to $e^{-2nA(\widehat{X}, t)}$, and thus $V(t)$ is proportional to $\int_{S^{2n-1}}e^{-2nA(\widehat{X}, t)}$. Claim \ref{clm:volume_integral} is proved.
		
		Next, we consider the following integral:
		$$\int_{B^{2n}} e^{-(2n-\epsilon) A(X,t)} d\lambda(X).$$
		Using a formula $A(X,t) = A(r\widehat{X},t) = A(\widehat{X},t)+ \log r$, we can calculate as
		\begin{align*}
		&\int_{S^{2n-1}} \int_0^1 e^{-(2n-\epsilon) A(\widehat{X}, t) - (2n-\epsilon)\log r} r^{2n-1}dr dS(\widehat{X})\\
		&= \int_{S^{2n-1}} e^{-(2n-\epsilon) A(\widehat{X}, t) }\left(\int_0^1 r^{-1 +\epsilon} dr \right) dS(\widehat{X}).
		\end{align*}
		Since $\int_0^1 r^{-1 +\epsilon} dr = \frac{1}{\epsilon}$, we have that
		$$ \int_{B^{2n}} e^{-(2n-\epsilon) A(X,t)} d\lambda(X) = \frac{1}{\epsilon}\int_{S^{2n-1}} e^{-(2n-\epsilon) A(\widehat{X}, t) }dS(\widehat{X}).$$
		The left-hand side is an integration of an $S^1$-invariant function (i.e.\ $A(e^{i \theta}X, t) = A(X, t)$).
		By \cite[Theorem 1.3, 1.]{Ber1},
		$$t \mapsto -\log\int_{S^{2n-1}} e^{-(2n-\epsilon) A(\widehat{X}, t) }dS(\widehat{X}) $$
		is subharmonic.
		
		Then let $\epsilon \to 0$.
		Since we assumed that $B^{2n} \Subset I_t$,  $A(\widehat{X}, t)$ is negative for all $\widetilde{X} \in S^{2n-1}$ and all $t$, and thus $e^{\epsilon A(\widehat{X}, t)}$ is increasing when $\epsilon \to 0$.
		Therefore $V(t)$ is a decreasing limit of subharmonic functions, and thus $V(t)$ is also subharmonic.
	\end{proof}
	
	\section{Balanced domains}
	In this section, we verify Conjecture \ref{conj:azukawa-psh} for balanced domains. Let $\Omega$ be a bounded pseudoconvex balanced domain and $h$ be the homogeneous plurisubharmonic function (i.e.\ $h(\lambda z) = |\lambda| h(z)$ for every $\lambda \in \mathbb{C}$) on $\mathbb{C}^n$ such that $\Omega = \{z \in \mathbb{C}^n: h(z)<1 \}$. Then we can describe the pluricomplex Green function and the Azukawa pseudometric for $(\Omega, 0)$ in terms of $h$ (cf.\ \cite{Zwo}):
	$$g_{\Omega,0} = \log h(z),\, A_{\Omega,0} = \log h(X).$$
	
	Therefore, to prove Theorem \ref{thm:balanced}, we shall only prove that $\log h_{t}(X)$ is plurisubharmonic in $(X,t)$.
	To prove that, we fix $(X,t)$ and prove subharmonicity on every complex line through $(X,t)$. Since $\log h_t(0) = -\infty$, the mean valur inequality is trivially satisfied when $X=0$. Thus we can assume $t=0$ and $X=(1,0,\ldots, 0)$ without loss of generality.
	
	Let $(\xi, \tau) \in \mathbb{C}^n \times \mathbb{C}$ be a vector representing the complex line. We shall show the subharmonicity of $\lambda \mapsto \log h_{\lambda \tau} ((1,0,\ldots, 0) + \lambda\xi)$. We divide the situation into three cases.
	
	{\it Case 1: $\tau = 0$.} In this case, the desired subharmonicity directly follows from the plurisubharmonicity of the Azukawa pseudometric. It is known to be plurisubharmonic for a fixed domain (cf. \cite{Zwo}).
	
	{\it Case 2: $\xi = 0$.} We shall show that
	$$t \mapsto \log h_t(1,0,\ldots, 0)$$
	is subharmonic.
	By \cite[Theorem 2.6.7]{Hor}, 
	$$-\log \mathrm{dist}_{z_1}((t,0), \partial \widetilde{\Omega}) = -\log \mathrm{dist}_{z_1}(0, \partial \Omega_t) $$
	is subharmonic with respect to $t$.
	Here $\mathrm{dist}_{z_1}$ means the distance in $z_1$-direction, i.e.\
	$$\mathrm{dist}_{z_1}(p, \partial \Omega) = \sup \{t>0: p + (\lambda, 0,\ldots,0) \in \Omega \text{ for all }|\lambda| < t \}. $$
	We also have that $h_t(1,0,\ldots,0) = -1 / \mathrm{dist}_{z_1}(0, \partial \Omega_t)$ and thus
	$$\log h_t(1,0, \ldots, 0) = - \log \mathrm{dist}_{z_1}(0, \partial \Omega_t).$$
	It is subharmonic.
	
	{\it Case 3: General cases.}
	We shall show that
	\begin{equation}
	t \mapsto \log h_t(1+t\xi_1,t\xi_2, \ldots, t\xi_n)\label{eqn:general-case}
	\end{equation}
	is subharmonic.
	Consider a function
	\begin{equation}
	(t, z_1, \ldots, z_n) \mapsto \log h_t(z_1 + tz_1\xi_1, z_2 + tz_1\xi_2, z_3 + tz_1\xi_3 , \ldots, z_n + tz_1 \xi_n).\label{eqn:general-case-2}
	\end{equation}
	Then the function (\ref{eqn:general-case}) is obtained from (\ref{eqn:general-case-2}) by substitution $(z_1, \ldots,z_n) = (1,0, \ldots,0)$. The function (\ref{eqn:general-case-2}) is a composition of a mapping
	$$\varphi_t(z_1, \ldots, z_n)  = \begin{pmatrix}
	z_1 \\
	z_2 \\
	\vdots\\
	z_n
	\end{pmatrix}+ t\begin{pmatrix}
	\xi_1 & 0 & \ldots & 0\\
	\xi_2 & 0 & \ldots & 0\\
	\vdots & \vdots & \vdots & \vdots\\
	\xi_n & 0 & \ldots & 0\\
	\end{pmatrix} \begin{pmatrix}
	z_1 \\
	z_2 \\
	\vdots\\
	z_n
	\end{pmatrix} $$
	and $\log h_t$.
	The mapping $\varphi_t$ is bihomolorphic with respect to $(t, z_1, \ldots, z_n)$ if $t\ll 0$ and linear isomorphism in $(z_1, \ldots, z_n)$ for each fixed $t$.
	Specifically, it maps a balanced domain $\Omega_t$ to another balanced domain. Thus, considering the coordinate change under $\varphi$, we can reduce the situation to Case 2.
	
	\section{An example}
	We shall prove Theorem \ref{thm:example} in this section. 
	
	We can compute the pluricomplex Green functions on a ball using M\"obius transformations \cite{Kli}:
\begin{equation}
g_{B_{0,R},w}(z) = \log |T_{w/R} (z/R)|. \label{eqn:pluricomplex}
\end{equation}
	Here, $T_{a}$ is a biholomorphic mapping on the unit ball in $\mathbb{C}^n$ defined by
	$$T_a(z) := \frac{a - P_a z -\sqrt{1-|a|^2} Q_a z}{a - \langle z, a \rangle},$$
	where $P_a$ denotes the orthogonal projection to the subspace $\mathbb{C} \cdot a \subset\mathbb{C}^n$ and $Q_a$ denotes the orthogonal projection to its orthogonal complement: $P_a + Q_a = \mathrm{Id}_{\mathbb{C}^n}$.
	
	From here we assume that $n=2$ and $a = (r,0)$ for $0<r < 1$. Then $T_a$ is written as follows:
	$$\twodimvector{z_1}{z_2} \mapsto \frac{1}{1-rz_1} \twodimvector{r-z_1}{\sqrt{1-r^2} z_2}.$$
	By the equation \ref{eqn:pluricomplex}, we have that
	$$g_{\mathbb{B}, (r,0)} = \frac{1}{2} \log \left(\frac{|r-z_1|^2 + (1-r^2)|z_2|^2}{|1-rz_1|^2}\right). $$
	The Azukawa pseudometrics is
	$$A_{\mathbb{B}, (r,0)} (X) = \lim_{\lambda \to 0} \left[g_{\mathbb{B}, (r,0)} \left( \twodimvector{r}{0} + \lambda \twodimvector{X_1}{X_2} \right) - \log |\lambda| \right] = \frac{1}{2} \log \frac{|X_1|^2 + (1-r^2) |X_2|^2}{(1-r^2)^2}.$$
	
	We note that $A_{\rho \Omega} (X) = A_\Omega (X) - \log \rho$ for $\rho > 0$.
	Using these formulas, if $\Omega_t = \{z \in \mathbb{C}^2: |z| < e^{-\varphi(t)}\}$ and $w = (1,0)$, then we have
	$$A_{\Omega_t, w} = \varphi(t) + \frac{1}{2}\log \frac{|X_1|^2 + (1-e^{2\varphi})|X_2|^2}{(1-e^{2\varphi})^2}.$$
	
	We rewrite the right-hand side as
	$$\frac{1}{2} \log \left[ \frac{1}{1-e^{2\varphi}} |X_1|^2 + |X_2|^2 \right] + \frac{1}{2} \log \frac{e^{2\varphi}}{1-e^{2\varphi}}.$$
	We need the following simple lemma:
	\begin{lemma}
		Let $P$ and $Q$ be upper semi-continuous functions on an open set in $\mathbb{C}^n$. If $\log P$ and $\log Q$ are both plurisubharmonic, then $\log(P+Q)$ is also plurisubharmonic.
	\end{lemma}
	
	\begin{proof}
		It follows from the convexity of the function $(x,y) \mapsto \log(e^x + e^y)$ on $\mathbb{R}^2$.
	\end{proof}
	By the lemma, it suffices to show the plurisubharmonicity of $\log \frac{1}{1-e^{2\varphi}}|X_1|^2 = \log \frac{1}{1-e^{2\varphi}} + \log |X_1|^2$ and $\log |X_2|^2$. $\log\frac{1}{1-e^{2\varphi}}$ is subharmonic in $t$ because $x \mapsto \log\frac{1}{1-e^{2x}}$ is increasing and convex. Similarly, $x \mapsto \log \frac{e^{2x}}{1-e^{2x}}$ is also increasing and convex, thus $\log \frac{e^{2\varphi}}{1-e^{2\varphi}}$ is subharmonic in $t$. Therefore $A_{\Omega_t, w}$ is plurisubharmonic as desired.
	
	\section{Relation to the Ohsawa-Takegoshi $L^2$-extension theorem }
	In this section, we treat the {\it Ohsawa-Takegoshi} $L^2$ extension theorem. We will work in the following setting. Let $\Omega \subset \mathbb{C}^n$ be a bounded pseudoconvex domain and $w \in \Omega$ be a point. Let $\varphi$ be a plurisubharmonic function on $\Omega $ with $\varphi(w) \neq -\infty$. In this setting, the Ohsawa-Takegoshi $L^2$ extension theorem can be stated as follows:
	
	\begin{theorem}[\cite{OT}]
		Under these settings, there exists a holomorphic function $F$ on $\Omega$ with F(w) = 1 and
		$$\int_\Omega |F|^2 e^{-\varphi} \leq C e^{-\varphi(w)},$$
		where $C$ is a constant which only depends on the dimension and diameter of the domain.
	\end{theorem}
	
	To find an optimal value of $C$ is a long-standing problem. B{\l}ocki \cite{Blo} and Guan-Zhou \cite{GZ} obtained such a optimal result, and Berndtsson-Lempert \cite{BL} gave an alternative proof, in which a variational result in \cite{Ber2} was used.
	Here we state an optimal $L^2$-extension theorem in the form \cite{BL}, for the case of extension from a point.
	
	\begin{theorem}[\cite{BL}]
	Let $\Omega \subset \mathbb{C}^n$ be a bounded pseudoconvex domain and $\phi$ be a plurisubharmonic function on $\Omega$.
	Let $w \in \Omega$. Let $G$ be a ``Green-type'' function with a pole at $w$, i.e.\ $G$ is a negative plurisubharmonic function on $\Omega$ such that there exist continuous functions $A$ and $B$ on $\Omega$ satisfying
	$$\log|z-w|^2 + A(z) \geq G(z) \geq \log|z-w|^2 - B(z).$$
	Then there exists a holomorphic function $F$ on $\Omega$ with $F(w) = 1$ and
	$$\int_\Omega |F|^2 e^{-\phi} \leq \sigma_n e^{-\phi(w)+ nB(w)}, $$
	where $\sigma_n$ be a volume of the unit ball in $\mathbb{C}^n$.
	\end{theorem}
	In this formulation, a Robin-type constant $B(w)$ has an important roll. 
	When the subvariety is a one-point set, replacing it by the volume of the Azukawa indicatrix, we can improve the result.
	\begin{theorem}\label{thm:optimal_Azukawa}
		Let $\Omega \subset \mathbb{C}^n$ be a bounded pseudoconvex domain and $w \in \Omega$ be a point.
		Let $g_{\Omega,w}$ be the pluricomplex Green function on $\Omega$ with a pole at $w$ and $A_{\Omega, w}$ be the corresponding Azukawa peudometric.
		We assume that $\limsup$ in the definition of $A_{\Omega,w}$ converges:
		$$A_{\Omega, w}(X) = \lim_{\lambda \to 0} g_{\Omega, w}(w + \lambda X)- \log |\lambda|$$
		for every $X \in \mathbb{C}^n$.
		Then, for every $\varphi \in PSH(\Omega)$, there exists a holomorphic function $F \in \mathcal{O}(\Omega)$ such that $f(w) = 1$ and
		$$\int_{\Omega} |F|^2 e^{-\varphi} \leq V(I_{\Omega,w})e^{-\varphi(0)},$$
		where $V(I_{\Omega, w})$ denotes the Euclidean volume of $I_{\Omega, w}$.
	\end{theorem}
	
	\begin{proof}
	    The proof is essentially the same as one in \cite{BL}.
		The only difference is that the ball in \cite[Lemma 3.3]{BL} will be replaced with the Azukawa indicatrix.
		Precisely, we use the following lemma:
		\begin{lemma}\label{lem:infinitesimal-integral}
			Let $\chi \geq 0$ be a continuous function on $\overline{\Omega}$.
			Then we have that
			$$\limsup_{t \to -\infty} e^{-kt} \int_{\Omega_t} \chi \leq V(I_{\Omega,w})\chi(w),$$
			where $\Omega_t := \{g_{\Omega, w} <t/2\}$.
		\end{lemma}
		
		
		\begin{proof}
			We can assume $w=0$.
			Since $\chi$ is continuous, it is sufficient to prove that $$\limsup_{t \to -\infty} (e^{-t/2}\Omega_t) \subset \overline{I_{\Omega, w}}.$$
			A point $z$ belongs to the set $\limsup_{t \to -\infty} (e^{-t/2}\Omega_t)$ if and only if for every $p<0$ there exists $t<p$ such that $z \in e^{-t/2} \Omega_t$.
			The last condition can be rewritten as
			\begin{align*}
			&e^{t/2} z \in \Omega_t\\
			\Longleftrightarrow \,& g_{\Omega,0}(e^{t/2}z) < t/2\\
			\Longleftrightarrow \,& g_{\Omega,0}(e^{t/2}z) - \log {e^{t/2}} < 0.
			\end{align*}
			By the assumption, the left-hand side of the last line converges to $A_{\Omega, 0}(z)$. Therefore $A_{\Omega,0} (z) \leq 0$ and thus $z \in \overline{I_{\Omega, w}}$. The lemma is proved.
			
		\end{proof}
		
		Using this lemma and repeating the argument in \cite{BL}, we can prove Theorem \ref{thm:optimal_Azukawa}.
	\end{proof}
	
	It may be interesting to generalize Theorem \ref{thm:optimal_Azukawa} to cases where subvarieties have positive dimension.
	To do that, we may have to consider the Azukawa pseudometric for pluricomplex Green function with poles along a subvariety.
	
	\vskip3mm
	{\bf Acknowledgment. } The author would like to thank Prof.\ Shigeharu Takayama and Prof.\ Takeo Ohsawa for valuable comments.

\bibliographystyle{plain}

\end{document}